\documentclass[11pt, reqno]{amsart}
\usepackage{indentfirst, amssymb, amsmath, amsthm, mathrsfs, setspace, indentfirst, enumerate,  mathrsfs, amsmath, amsthm}
\usepackage[bookmarksnumbered, colorlinks, plainpages]{hyperref}
\usepackage{mathrsfs}

\newtheorem*{theoA}{Theorem A}
\newtheorem*{theoB}{Theorem B}
\newtheorem*{theoC}{Theorem C}
\newtheorem*{theoD}{Theorem D}
\newtheorem*{theoE}{Theorem E}

\newtheorem*{remA}{Remark A}

\newtheorem*{cor A}{Corollary A}
\newtheorem*{cor B}{Corollary B}
\newtheorem{ques}{Question}[section]
\newtheorem{theo}{Theorem}[section]
\newtheorem{lem}{Lemma}[section]

\newtheorem{exm}{Example}[section]

\newtheorem{rem}{Remark}[section]

\newcommand{\ol}{\overline}
\newcommand{\be}{\begin{equation}}
\newcommand{\ee}{\end{equation}}
\newcommand{\beas}{\begin{eqnarray*}}
\newcommand{\eeas}{\end{eqnarray*}}
\newcommand{\bea}{\begin{eqnarray}}
\newcommand{\eea}{\end{eqnarray}}

\numberwithin{equation}{section}
\begin{document}

\title[U\MakeLowercase {niqueness of entire functions}......]{\LARGE U\Large\MakeLowercase {niqueness of entire functions sharing two values with their partial derivative operators}}

\date{}
\author[S. M\MakeLowercase {ajumder}, D. P\MakeLowercase {ramanik} \MakeLowercase{and} S. P\MakeLowercase{anja}]{S\MakeLowercase {ujoy} M\MakeLowercase {ajumder}, D\MakeLowercase {ebabrata} P\MakeLowercase {ramanik} \MakeLowercase{and} S\MakeLowercase{hantanu} P\MakeLowercase{anja}$^*$}
\address{Department of Mathematics, Raiganj University, Raiganj, West Bengal-733134, India.}
\email{sm05math@gmail.com, sjm@raiganjuniversity.ac.in}
\address{Department of Mathematics, Raiganj University, Raiganj, West Bengal-733134, India.}
\email{debumath07@gmail.com}
\address{Department of Mathematics, University of Kalyani, West Bengal 741235, India.}
\email{panjasantu07@gmail.com}

\renewcommand{\thefootnote}{}
\footnote{*\emph{Corresponding Author}: Shantanu Panja.}

\renewcommand{\thefootnote}{\arabic{footnote}}

\setcounter{footnote}{0}

\begin{abstract} In this paper, we employ the theory of normal families in several complex variables to obtain some uniqueness theorems for entire functions. These results extend the related works of Li and Yi \cite{Li-Yi-2006}, and L\"{u} et al. \cite{Lu-Xu-Yi-2009} to the setting of several complex variables. Moreover, some examples are provided to demonstrate the sharpness of our results.
\end{abstract}
\thanks{Typeset by \AmS -\LaTeX}
\maketitle

\section{{\bf Introduction and main results}}
The study of entire functions that share values with their partial derivatives is a classical and active topic in complex analysis, particularly within Nevanlinna theory. When an entire function partially shares two values with its partial derivatives, the behavior of the function becomes significantly restricted. Such restrictions often lead to strong uniqueness results or explicit characterizations of the entire function.

\smallskip
Entire functions sharing values with their partial derivatives arise naturally in the study of differential equations, normal families, and uniqueness theory. Investigating partial sharing such as IM-sharing, CM-sharing, or sharing with reduced multiplicity-extends classical uniqueness theorems while capturing more subtle analytical phenomena. These generalizations provide deeper insight into the structural properties of entire functions and their derivatives in several complex variables.
 
\medskip
Let $f(z)$ and $g(z)$ be two non-constant meromorphic functions and $a$ be a finite complex number. If $g(z)-a=0$ whenever $f(z)-a=0$, we write $f(z)=a\Rightarrow g(z)=a$. If $f(z)=a\Rightarrow g(z)=a$ and $g(z)=a\Rightarrow f(z)=a$, we then write $f(z)=a\Leftrightarrow g(z)=a$ and we say that $f(z)$ and $g(z)$ share $a$ IM. If $f(z)-a$ and $g(z)-a$ have the same zeros with the same multiplicities, we write $f(z)=a\rightleftharpoons g(z)=a$ and we say that $f(z)$ and $g(z)$ share $a$ CM.\par

\medskip
In 1977, Rubel and Yang \cite{Rubel-Yang-1977} were the first to investigate the uniqueness problem for an entire function $f(z)$ that shares two values with its first derivative $f^{(1)}(z)$. They established the following result.

\begin{theoA}\cite{Rubel-Yang-1977} Let $f(z)$ be a non-constant entire function in $\mathbb{C}$ and let $a$ and $b$ be two distinct finite complex numbers. If $f(z)=a\rightleftharpoons f^{(1)}(z)=a$ and $f(z)=b\rightleftharpoons f^{(1)}(z)=b$, then $f(z)\equiv f^{(1)}(z)$, i.e., $f(z)=A\exp(z)$, where $A\in\mathbb{C}\backslash \{0\}$.
\end{theoA} 

\medskip
In 1979, Mues and Steinmetz \cite{Mues-Steinmetz-1979} further generalized Theorem A by weakening the sharing condition from CM to IM, as stated in the following result.
\begin{theoB}\cite[Satz 1]{Mues-Steinmetz-1979} Let $f(z)$ be a non-constant entire function in $\mathbb{C}$ and let $a$ and $b$ be two distinct finite complex numbers. If $f(z)=a\Leftrightarrow f^{(1)}(z)=a$ and $f(z)=b\Leftrightarrow f^{(1)}(z)=b$, then $f(z)\equiv f^{(1)}(z)$, i.e., $f(z)=A\exp(z)$, where $A\in\mathbb{C}\backslash \{0\}$.
\end{theoB} 

These two results ushered in a new era in the study of uniqueness problems for entire and meromorphic functions sharing two values with their derivatives, and they initiated a long-standing line of research in this field.

Essentially, Theorems A and B have been generalized in two main directions. The first direction replaces the sharing condition $f(z)=a\rightleftharpoons f^{(1)}(z)=a$ with the weaker requirement $f(z)=a\Rightarrow f^{(1)}(z)=a$. The second direction reverses the implication by replacing $f(z)=a\rightleftharpoons f^{(1)}(z)=a$ with $f^{(1)}(z)=a\Rightarrow f(z)=a.$

\smallskip
In this paper, we refer to the following result, due to Li and Yi \cite{Li-Yi-2006}, in which the original sharing condition $f(z)=b\rightleftharpoons f^{(1)}(z)=b$ is weakened using the concept of partial value sharing, namely $f(z)=a\Rightarrow f^{(1)}(z)=a$ and $f^{(1)}(z)=a\Rightarrow f(z)=a$ respectively.

\begin{theoC}\cite[Theorem 1]{Li-Yi-2006} Let $f(z)$ be a non-constant entire function in $\mathbb{C}$ and let $a$ and $b$ be two finite complex numbers such that $b\not=a,0$. If $f^{(1)}(z)=a\Rightarrow f(z)=a$ and $f^{(1)}(z)=b\rightleftharpoons f(z)=b$, then $f(z)\equiv f^{(1)}(z)$, i.e., $f(z)=A\exp(z)$, where $A\in\mathbb{C}\backslash \{0\}$.
\end{theoC} 

\begin{theoD}\cite[Theorem 2]{Li-Yi-2006} Let $f(z)$ be a non-constant entire function in $\mathbb{C}$ and let $a$ and $b$ be two finite complex numbers such that $b\not=a, 0$. If $f(z)=a\Rightarrow f^{(1)}(z)=a$ and $f(z)=b\rightleftharpoons f^{(1)}(z)=b$, then one of the following cases must occur. 
\begin{enumerate} 
\item[(1)] $f(z)\equiv f^{(1)}(z)$, 
\item[(2)] $f(z)=c\exp(\frac{b}{b-a}z)+a$, where $c\in\mathbb{C}\backslash \{0\}$. 
\end{enumerate}  
\end{theoD}

\begin{remA} We observe that $b\not=a$ in Theorem C. Therefore, if we assume $b\not=0$ in Theorem A, then it is clear that Theorem C generalizes Theorem A through the notion of ``partial'' sharing of values.
\end{remA}

It would be a remarkable achievement to show that the conclusion of Theorem C remains valid when the hypothesis $f(z)=b\rightleftharpoons f^{(1)}(z)=b$ is replaced by the weaker condition $f(z)=b\Leftrightarrow f^{(1)}(z)=b$.
In 2009, L\"{u} et al. \cite{Lu-Xu-Yi-2009} accomplished this by establishing the following result. 

\begin{theoE}\cite[Theorem 1.1]{Lu-Xu-Yi-2009} Let $f(z)$ be a non-constant entire function in $\mathbb{C}$ and let $a$ and $b$ be two non-zero finite complex numbers such that $b\not=a$. If $f^{(1)}(z)=a\Rightarrow f(z)=a$  and $f(z)=b\Leftrightarrow f^{(1)}(z)=b$, then $f(z)\equiv f^{(1)}(z)$, i.e., $f(z)=A\exp(z)$, where $A\in\mathbb{C}\backslash \{0\}$.
\end{theoE}

In one complex variable, many uniqueness theorems for meromorphic functions sharing values with their derivatives follow directly from classical Nevanlinna theory. However, when we move to several complex variables, the situation becomes more delicate because:

$\bullet$ Derivatives become partial or directional derivatives rather than a single derivative $f^{(1)}$.

$\bullet$ The classical Nevanlinna theory must be replaced by its higher-dimensional generalizations for meromorphic functions or mappings in $\mathbb{C}^n$.

$\bullet$ Additional geometric conditions (such as directional derivatives, growth conditions, or linearly nondegenerate mappings) are often required.

\medskip
In 1995, Berenstein et al. \cite{Berenstein-Chang-Li-1995} obtained a Nevanlinna type uniqueness result for non-constant meromorphic functions in $\mathbb{C}^n$ and their directional derivatives in terms of shared values. They proved: Let $f$ be a non-constant meromorphic function in $\mathbb{C}^n$. If $f$ and the directional derivative $\partial_u(f)$ of $f$ along a direction $u\in S^{2n-1}$ share three distinct values in $\mathbb{C}\cup\{\infty\}$ counting multiplicities, then $f\equiv \partial_u(f)$, where 
\begin{align*}
 S^{2n-1}&=\{(u_1,u_2,\ldots,u_n)\in C^{n}: |u_1|^2+|u_2|^2+\cdots+|u_n|^2=1\}\\
\partial_u(f)&= \sum_{j=1}^{n} u_{j}\frac{\partial f}{\partial z_j},\;\; u=(u_1, \ldots ,u_n).
\end{align*}

This result was extended to moving targets in $\mathbb{C}^n$ by Hu and Yang \cite{Hu-Yang-1996} in 1996.

We define $\mathbb{Z}_+=\mathbb{Z}[0,+\infty)=\{n\in \mathbb{Z}: 0\leq n<+\infty\}$ and $\mathbb{Z}^+=\mathbb{Z}(0,+\infty)=\{n\in \mathbb{Z}: 0<n<+\infty\}$.
On $\mathbb{C}^n$, we define the following operators
\[\partial_{z_i}=\frac{\partial}{\partial z_i},\partial^2_{z_jz_i}=\frac{\partial^2}{\partial z_jz_i},\ldots, \partial_{z_i}^{l_i}=\frac{\partial^{l_i}}{\partial z_i^{l_i}}\;\;\text{and}\;\;\partial^{I}=\frac{\partial^{|I|}}{\partial z_1^{i_1}\cdots \partial z_m^{i_m}}\]
where $l_i\in \mathbb{Z}^+\;(i=1,2,\ldots,n)$ and $I=(i_1,\ldots,i_n)\in\mathbb{Z}^n_+$ is a multi-index such that $|I|=\sum_{j=1}^n i_j$.

\medskip
Theorems C-E lead to the following natural question:
\begin{ques} Do Theorems C-E remain valid in the setting of several complex variables with respect to partial derivative operators?
\end{ques}

We respond to Question 1.1 in the paper with the following outcome.

\begin{theo}\label{t1.1} Let $f(z)$ be a transcendental function in $\mathbb{C}^n$ and let $a$ and $b$ be two finite complex numbers  in $\mathbb{C}$ such that $b\not=a, 0$. If 
\begin{enumerate}
\item[(i)] $f=a\Rightarrow \partial_{z_i}(f)=a$ for $i=1,2,\ldots,n$,
\item[(ii)] $f=b\Leftrightarrow \partial_{z_i}(f)=b$ for $i=1,2,\ldots,n$, 
\end{enumerate}
then one of the following cases must occur:
\begin{enumerate} 
\item[(1)] $f(z)=ce^{z_1+z_2+\ldots+z_n}$, where $c$ is a non-zero constant in $\mathbb{C}$, 
\item[(2)] $f(z)=a+ce^{\frac{b}{b-a}(z_1+z_2+\ldots+z_n)}$, where $c$ is a non-zero constant in $\mathbb{C}$,
\item[(3)] $f(z)=b+ce^{\frac{a}{b-a}(z_1+z_2+\ldots+z_n)}$, where $a\neq 0$ and $c$ is a non-zero constant in $\mathbb{C}$.
\end{enumerate}
\end{theo}

\begin{rem} The following example demonstrates that, for the validity of Theorem \ref{t1.1}, it is necessary that the conditions $(i)$ and $(ii)$ hold for every $i=1,2,\ldots,n$.
\end{rem}
\begin{exm} Let $f(z)=a+ce^{\frac{b}{b-a}(z_1+z_2^2+\ldots+z_n^2)}$, where $a$, $b$ and $c$ are non-zero constants in $\mathbb{C}$ such that $a\neq b$. It is easy to verify that the conditions $f=a\Rightarrow \partial_{z_1}(f)=a$ and $f=b\Leftrightarrow \partial_{z_1}(f)=b$ both hold, but the condition $f=b\Leftrightarrow \partial_{z_i}(f)=b$ does not hold for $i=2,3,\ldots,n$.
\end{exm}

Now we state the multidimensional version of Theorem E.

\begin{theo}\label{t1.2} Let $f(z)$ be a transcendental function in $\mathbb{C}^n$ and let $a$ and $b$ be two finite non-zero complex numbers in $\mathbb{C}$ such that $a\not=b$. If 
\begin{enumerate}
\item[(i)] $\partial_{z_i}(f)=a\Rightarrow f=a$ for $i=1,2,\ldots,n$,
\item[(ii)] $f=b\Leftrightarrow \partial_{z_i}(f)=b$ for $i=1,2,\ldots,n$, 
\end{enumerate}
then $f(z)=ce^{z_1+z_2+\ldots+z_n}$, where $c$ is a non-zero constant in $\mathbb{C}$.
\end{theo}

Next we state the multidimensional version of Theorem C.
\begin{theo}\label{t1.3} Let $f(z)$ be a transcendental function in $\mathbb{C}^n$ and let $b$ be a finite non-zero complex numbers in $\mathbb{C}$. If 
\begin{enumerate}
\item[(i)] $\partial_{z_i}(f)=0\Rightarrow f=0$ for $i=1,2,\ldots,n$,
\item[(ii)] $f=b\rightleftharpoons \partial_{z_i}(f)=b$ for $i=1,2,\ldots,n$, 
\end{enumerate}
then one of the following cases must occur:
\begin{enumerate} 
\item[(1)] $f(z)=ce^{z_1+z_2+\ldots+z_n}$, where $c$ is a non-zero constant in $\mathbb{C}$, 
\item[(2)] $f(z)=\frac{c}{A}e^{A(z_1+z_2+\ldots+z_n)}+b-\frac{b}{A}$, where $A$ and $c$ are non-zero constants in $\mathbb{C}$.
\end{enumerate}
\end{theo}

\begin{rem} The following example shows that the number of shared values cannot be reduced to one in Theorem \ref{t1.1}.
\end{rem}

\begin{exm} Let 
\begin{align*}
f(z_1,\ldots,z_n)=e^{e^{z_1+\cdots+z_n}}\int_{0}^{z_1+\cdots+z_n}e^{-e^t}(1-e^t)dt.
\end{align*}

Note that for all $i=1,2,\ldots,n$, we have
\begin{align*}
\partial_{z_i}(f(z))=e^{z_1+\cdots+z_m}(f(z)-1)+1
\end{align*}
and so for all $i=1,2,\ldots,n$, we get
\begin{align*}
\partial_{z_i}(f(z))-1=e^{z_1+\cdots+z_n}(f(z)-1).
\end{align*}

Clearly $f(z)=1 \rightleftharpoons \partial_{z_i}(f(z))=1$ for $i=1,2,\ldots,n$, but $f(z)$ does not satisfy any case of Theorem \ref{t1.1}.
\end{exm}

\begin{rem}
Following example shows that \textrm{Theorem \ref{t1.1}} does not hold in general for a non-constant meromorphic function $f(z)$ in $\mathbb{C}^n$.
\end{rem}

\begin{exm} Let 
\[f(z)=\frac{z_1+z_2+\ldots+z_n}{1-e^{-(z_1+z_2+\ldots+z_n)}},\]
$a=0$ and $b=1$. Clearly 
\begin{align*}
\frac{\partial f(z)}{\partial z_i}=\frac{1-(z_1+z_2+\ldots+z_n+1)e^{-(z_1+z_2+\ldots+z_n)}}{\left(1-e^{-(z_1+z_2+\ldots+z_n)}\right)^2}
\end{align*}
and so $f(z)=a \Rightarrow \partial_{z_i}(f(z))=a$ for $i=1,2,\ldots,n$. Note that
\begin{align*}
f(z)-b=\frac{z_1+z_2+\ldots+z_n-1+e^{-(z_1+z_2+\ldots+z_n)}}{1-e^{-(z_1+z_2+\ldots+z_n)}}
\end{align*}
and
\begin{align*}
\partial_{z_i}(f(z))-b=-e^{-(z_1+z_2+\ldots+z_n)}\;\;\frac{z_1+z_2+\ldots+z_n-1+e^{-(z_1+z_2+\ldots+z_n)}}{(1-e^{-(z_1+z_2+\ldots+z_n)})^{2}}.
\end{align*}

Clearly $f(z)=b \rightleftharpoons \partial_{z_i}(f(z))=b$ for $i=1,2,\ldots,n$, but $f(z)$ does not satisfy any case of Theorem \ref{t1.1}.
\end{exm}

\section{\bf{Basic Notations in several complex variables}}
Nevanlinna theory, originally formulated for meromorphic functions of a single complex variable, has been extended to the setting of several complex variables to study the value distribution of meromorphic maps from $\mathbb{C}^n$ into complex projective spaces or complex manifolds. It provides profound insights into: $(a)$ The growth and value distribution of meromorphic mappings; $(b)$ Uniqueness and rigidity phenomena for solutions of linear and nonlinear PDEs and $(c)$ Criteria for compactness and convergence in families of meromorphic maps. Its applications span complex analysis, partial differential equations, complex geometry, and mathematical physics, offering both powerful techniques and elegant results. These references \cite{Banerjee-Majumder-2026}, \cite{Cao-Korhonen-2016}, \cite{PVD1}-\cite{PVD3}, \cite{BQL1}-\cite{FL1}, \cite{MDP}, \cite{Majumder-2026}, \cite{Majumder-Sarkar-2026}, \cite{Majumder-Sarkar-2027}, \cite{GS2} provide a foundation for understanding the current state of research in Nevanlinna value distribution theory in several complex variables.

\medskip
Let $G\neq \varnothing$ be an open subset of $\mathbb{C}^n$. Let $f$ be a holomorphic function in $\mathbb{C}^n$.
Take $a\in G$. Let $G_a$ be the connectivity component of $G$ containing $a$. Assume $f\mid_{G_a}\not\equiv 0$. Then a series
$f(z)=\sum_{\lambda=p}^{\infty}P_{\lambda}(z-a)$ converges on some neighborhood of $a$ and represents $f$ on this neighborhood. Here 
$P_{\lambda}$ is a homogeneous polynomial of degree $\lambda$ and $P_{p}\not\equiv 0$. The polynomials $P_{\lambda}$ depend on $f$ and $a$ only. The number $\mu^0_f(a)=p$ is called the zero multiplicity of $f$ at $a$ (see \cite[pp. 12]{Stoll-1974}).

\medskip
Let $f$ be a meromorphic function on $G$.
Take $a\in G$ and $c\in\mathbb{C}\cup\{\infty\}$. Let $G_a$ be the component of $G$ containing $a$. If $0\equiv f\mid_{G_a}\not\equiv c$, define $\mu^c_f(a)=0$. Assume $0\not\equiv f\mid_{G_a}\not\equiv c$. Then an open connected neighborhood $U$ of $a$ in $G$ and holomorphic functions $g\not\equiv 0$ and $h\not\equiv 0$ exist on $U$ such that $h. f\mid_U=g$ and $\dim g^{-1}(0)\cap h^{-1}(0)\leq n-2$, where $n=\dim(\mathbb{C}^n)$. Therefore the $c$-multiplicity of $f$ is just $\mu^c_f=\mu^0_{g-ch}$ if $c\in\mathbb{C}$ and $\mu^c_f=\mu^0_h$ if $c=\infty$. The function $\mu^c_f:G\to \mathbb{Z}$ is nonnegative and is called the $c$-divisor of $f$ (see \cite[pp. 12]{Stoll-1974}).
If $f\not\equiv 0$ on each component of $G$, then $\nu=\mu_f=\mu^0_f-\mu^{\infty}_f$ is called the divisor of $f$. The function $f$ is holomorphic on $G$ if and only if $\mu_f\geq 0$. We define $\displaystyle \operatorname{supp}\; \nu=\ol{\{z\in G: \nu(z)\neq 0\}}$.

\medskip
If $A\subseteq \mathbb{C}^n$ and $r\geq 0$, we define (see \cite[pp. 6]{Stoll-1974})
$A[r]=\{z\in A: ||z||\leq r\}$, $A(r)=\{z\in A: ||z||<r\}$, $A\langle r\rangle=\{z\in A: ||z||=r\}$ and $\tau:W\to \mathbb{R}_+$ by $\tau(z)=||z||^2$.
On $\mathbb{C}^n$, the exterior derivative $d$ splits $d= \partial+ \bar{\partial}$ and twists to $d^c= \frac{\iota}{4\pi}\left(\bar{\partial}- \partial\right)$. Clearly $dd^{c}= \frac{\iota}{2\pi}\partial\bar{\partial}$. The standard Kaehler metric on $\mathbb{C}^n$ is given by $\upsilon=dd^c\tau>0$. On $\mathbb{C}^n\backslash  \{0\}$, we define
$\displaystyle \omega=dd^c\log \tau\geq 0$ and $\sigma=d^c\log \tau \wedge \omega^{n-1}$, where $n=\dim(\mathbb{C}^n)$ (see \cite[pp. 6]{Stoll-1974}). 

For $t>0$, the counting function $n_{\nu}$ is defined by
\begin{align*}
 n_{\nu}(t)=t^{-2(n-1)}\int_{A[t]}\nu \upsilon^{n-1},
 \end{align*}
where $A=\text{supp}\;\nu$. The valence function of $\nu$ is defined by
\[N_{\nu}(r)=N_{\nu}(r,r_0)=\int_{r_0}^r n_{\nu}(t)\frac{dt}{t}\;\;(r\geq r_0).\]

\medskip
Also we write $N_{\mu_f^a}(r)=N(r,a;f)$ if $a\in\mathbb{C}$ and $N_{\mu_f^a}(r)=N(r,f)$ if $a=\infty$.

\smallskip
With the help of the positive logarithm function, we define the proximity function of $f$ by
\begin{align*}
m(r, f)=\mathbb{C}^n\langle r; \log^+ |f| \rangle=\int_{\mathbb{C}^n\langle r\rangle} \log^+ |f|\;\sigma.
\end{align*}

The characteristic function of $f$ is defined by $T(r, f)=m(r,f)+N(r,f)$.
We define $m(r,a;f)=m(r,f)$ if $a=\infty$ and $m(r,a;f)=m(r,1/(f-a))$ if $a$ is finite complex number. If $a\in\mathbb{C}$, then the first main theorem states that $m(r,a;f)+N(r,a;f)=T(r,f)+O(1)$, where $O(1)$ denotes a bounded function when $r$ is sufficiently large. We define the order of $f$ by
\begin{align*}
\rho(f):=\limsup _{r \rightarrow \infty} \frac{\log T(r, f)}{\log r}.
\end{align*}

Let $U\subset \mathbb{C}^n$ be an open subset. A closed subset $A\subset U$ is said to be analytic,
if for each $a\in A$, there are a finite number of holomorphic functions $f_1(z),f_2(z),\ldots,f_l(z)$ defined in a neighbourhood $N(a)$ of $a$ such that (see \cite[pp. 42]{NW})
\begin{align*}
A\cap N(a)=\{z\in N(a):f_1(z)=f_2(z)=\cdots=f_l(z)=0\}.
\end{align*}

If $f_j(z),1\leq j\leq l$ can be taken so that their differentials at $a\in A$, $df_1(a), df_2(a), \ldots,df_l(a)$ are linearly independent, then $a$ is called a regular point of $A$. A point of $A$ which is not regular called a singular point. The subset of all singular points of $A$ is denoted by $S(A)$. Set $R(A)=A\backslash  S(A)$. If $S(A)=\varnothing$, then $A$ is said to be regular.

\section{\bf{Auxiliary Lemmas}}
For every function $\varphi$ of class $C^2(\Omega)$, we define at each point $z\in\Omega$ an hermitian form (see \cite{PVD1,PVD2})
\begin{align*}
L_z(\varphi,\nu)=\sideset{}{_{l,k=1}^n}{\sum} \frac{\partial^2 \varphi(z)}{\partial z_k \partial \ol{z}_l}\nu_k \ol{\nu}_l,
\end{align*}
which is called the Levi form of the function $\varphi$. For a holomorphic function $f$ in $\Omega$, we define
\begin{align*}
f^{\#}(z)=\max_{||\nu||=1}\sqrt{L_z(\log (1+|f(z)|^2),\nu)}
\end{align*}
where $\nu=(\nu_1,\ldots,\nu_n)$. This quantity is well defined, since the Levi form $L_z(\log (1+|f(z)|^2),\nu)$ is non-negative for all $z\in\Omega$. Let ${\bf \nabla} f=(f_{z_1},\ldots,f_{z_n})$.

Applying Cauchy-Schwarz inequality, it is easy to prove that (see \cite[Remark 1]{ZZY1})
\begin{align*}
f^{\#}(z)=\sup_{||\nu||=1}\frac{|\langle {\bf \nabla} f(z),\nu\rangle|}{1+|f(z)|^2}=\frac{||{\bf \nabla} f(z)||}{1+|f(z)|^2}=\frac{\sqrt{\sum_{j=1}^n|f_{z_j}(z)|^2}}{1+|f(z)|^2},\;\;z\in\Omega,
\end{align*}
where $\langle z, w\rangle=\sum_{j=1}^n z_j\ol{w}_j$ is the Hermitian scalar product for $z=(z_1,\ldots,z_n)\in\mathbb{C}^n$ and $w=(w_1,\ldots,w_n)\in\mathbb{C}^n$.
\medskip

A family $\mathcal{F}$ of holomorphic functions on a domain $\Omega\subset \mathbb{C}^n$ is normal in $\Omega$ if every sequence of functions $\{f_j\}\subseteq \mathcal{F}$ contains either a subsequence which converges to a limit function $f\neq \infty$ uniformly on each compact subset of $\Omega$, or a subsequence which converges uniformly to $\infty$ on each compact subset.

A family $\mathcal{F}$ is said to be normal at a point $z_0\in \Omega$ if it is normal in some neighbourhood of $z_0$. A family of analytic functions $\mathcal{F}$ is normal in a domain $\Omega$ if and only if $\mathcal{F}$ is normal at each point of $\Omega$.

\medskip
We recall the Marty's characterization of normal families in terms of the spherical metric.

\begin{lem}\label{ln2.1}\cite[Theorem 2.1]{PVD1} A family $\mathcal{F}$ of functions holomorphic on $\Omega$ is normal on $\Omega\subset \mathbb{C}^n$ if and only if for each compact subset $K\subset \Omega$, there exists $M(K)$ such that at each point $z\in K$,
$f^{\#}(z)\leq M(K)$ $\forall$ $f\in\mathcal{F}$.
\end{lem} 

In 2021, Dovbush \cite{PVD2} obtained an appropriate generalization of Zalcman's Lemma \cite{LZ1} to several complex variables and the result is as follows
\begin{lem}\label{ln2.1a} \cite[Theorem 1.1]{PVD2} Suppose that a family $\mathcal{F}$ of functions holomorphic on $\Omega\subset \mathbb{C}^n$ is not normal at some point $z_0\in\Omega$. Then there exist sequences $f_j\in\mathcal{F}$, $z_j\to z_0$, $\rho_j\to 0$ such that the sequence
\[g_j(z):=\rho_j^{-\alpha}f_j(z_j+\rho_jz)\;\;(0\leq \alpha<1\;\;\text{arbitrary})\]
converges locally uniformly in $\mathbb{C}^n$ to a non-constant entire function $g$ satisfying $g^{\#}(z)\leq g^{\#}(0)=1$.
\end{lem}

In \cite{Majumder-2026}, Majumder studied the relation between the maximum modulus and the spherical metric of a holomorphic function in $\mathbb{C}^n$ and obtained the following result.
\begin{lem}\label{ln2.2}\cite{Majumder-2026} Let $f$ be a holomorphic function in $\mathbb{C}^n$. If $f^{\#}$ is bounded on $\mathbb{C}^n$, then $\rho(f)\leq 1$.
\end{lem}

\begin{lem}\label{ln2.3} \cite{Majumder-2026} Let $\mathcal{F}$ be a family of holomorphic functions on a domain $\Omega\subset \mathbb{C}^n$ and let $a$ and $b$ be two distinct finite complex numbers. For each $f\in \mathcal{F}$, if $f=a\Rightarrow \partial_{z_i}(f)=a$ and $f=b\Leftrightarrow \partial_{z_i}(f)=b$ for $i=1,2,\ldots,n$, then $\mathcal{F}$ is normal in $\Omega$.
\end{lem}

\begin{lem}\label{ln2.6}\cite[Lemma 3.59]{Hu-Li-Yang-2003} Let $P$ be a non-constant entire function in $\mathbb{C}^n$. Then 
\begin{align*}
\rho(e^P)=
\begin{cases}
\deg(P), & \text{if $P$ is a polynomial,}\\
+\infty, & \text{otherwise.}
\end{cases}
\end{align*}
\end{lem}

\begin{lem}\label{ln2.3a} Let $\mathcal{F}$ be a family of holomorphic functions on a domain $\Omega\subset \mathbb{C}^n$ and let $a$ and $b$ be two distinct finite complex numbers. For each $f\in \mathcal{F}$, if $f=a\Rightarrow \partial_{z_i}(f)=a$ and $\partial_{z_i}(f)=b\Rightarrow f=b$ for $i=1,2,\ldots,n$, then $\mathcal{F}$ is normal in $\Omega$.
\end{lem}

\begin{proof}
Since normality is a local property, it is enough to show that $\mathcal{F}$ is normal at each point $z_0\in \Omega$.
Suppose on the contrary that $\mathcal{F}$ is not normal at $z_0\in\Omega$. Set $\mathcal{F}_a=\{f-a: f\in\mathcal{F}\}$. Then $\mathcal{F}_a$ is not normal in $\Delta$. Consequently, we may assume that $\mathcal{F}_a$ is not normal at $z_0\in\Delta$. 
Then by Lemma \ref{ln2.1a}, there exist a sequence of functions $f_j\in\mathcal{F}$, a sequence $\{z_j\}\subset\Omega$ with $z_j\rightarrow z_0$ and a sequence of positive numbers $\{\rho_j\}$ with $\rho_j\rightarrow 0$ such that
\begin{align}\label{rj1.1} 
F_j(\zeta)=\rho_j^{-\frac{1}{2}}\left\lbrace f_j(z_j+\rho_j \zeta)-a\right\rbrace\rightarrow F(\zeta)
\end{align}
locally uniformly in $\mathbb{C}^n$ to a non-constant holomorphic function $F$ such that $F^{\#}(\zeta)\leq F^{\#}(0)=1$, $\forall$ $\zeta\in\mathbb{C}^n$. Since $F^{\#}(\zeta)\leq 1$ for all $\zeta\in\mathbb{C}^n$, by Lemma \ref{ln2.2}, we get $\rho(F)\leq 1$.  On the other hand from the proof of Theorem 1.1 \cite{PVD2}, we get
\begin{align}
\label{rj1.2} \rho_{j}=\frac{1}{g_{j}^{\#}(z_{j})},
\end{align}
where $g_j(z_j)=f_{j}(z_{j})-a$. Also (\ref{rj1.1}) gives
\begin{align}\label{rj1.3} 
\partial_{\zeta_i}(F_j(\zeta))=\rho_j^{\frac{1}{2}}\partial_{\zeta_i}(f_j(z_j+\rho_j \zeta))\rightarrow \partial_{\zeta_i}(F(\zeta)),
\end{align}
locally uniformly in $\mathbb{C}^n$, where $i\in\{1,2,\ldots,n\}$. 
We consider the following two cases.\par

\smallskip
{\bf Case 1.} Let $0$ be a Picard exceptional value of $F$. Since $\rho(F)\leq 1$, using Lemma \ref{ln2.6}, we may assume that 
\begin{align}
\label{rj1.4} F(\zeta)=Ae^{\lambda_1 \zeta_1+\lambda_2\zeta_2+\ldots+\lambda_n\zeta_n},
\end{align}
where $A(\neq 0)$ and $\lambda_1,\lambda_2,\ldots,\lambda_n$ are constants such that $(\lambda_1,\lambda_2,\ldots,\lambda_n)\neq (0,0,\ldots,0)$. Clearly from (\ref{rj1.1}), (\ref{rj1.3}) and (\ref{rj1.4}), we have 
\begin{align*}
\frac{\partial_{\zeta_i}(F_j(\zeta))}{F_j(\zeta)}=\rho_{j}\frac{\partial_{\zeta_i}(f_j(z_j+\rho_j \zeta))}{f_j(z_j+\rho_j \zeta)-a}\rightarrow \frac{\partial_{\zeta_i}(F(\zeta))}{F(\zeta)}=\lambda_i,
\end{align*}
locally uniformly in $\mathbb{C}^n$, where $i\in\{1,2,\ldots,n\}$. Consequently, we have
\begin{align}
\label{rj1.5} \rho_{j}\frac{\sqrt{\sum\limits_{i=1}^n|\partial_{\zeta_i}(f_j(z_j))|^2}}{|f_j(z_j)-a|}\rightarrow \sqrt{\sum\limits_{i=1}^n|\lambda_i|^2}.
\end{align}

Now using (\ref{rj1.2}) to (\ref{rj1.5}), we get
\begin{align*}
\frac{1+|f_j(z_j)-a|^2}{|f_j(z_j)-a|}\rightarrow \sqrt{\sum\limits_{i=1}^n|\lambda_i|^2}\neq 0,
\end{align*}
which shows that
\begin{align*}
\lim\limits_{j\to \infty}\left\lbrace f_j(z_j)-a\right\rbrace\neq 0, \infty
\end{align*}
and so from (\ref{rj1.1}), we have
\begin{align*} 
F_j(0)=\rho_n^{-\frac{1}{2}}\left\lbrace f_j(z_j)-a\right\rbrace\to \infty.
\end{align*}

Again, from (\ref{rj1.1}) and (\ref{rj1.4}), we have $F_j(0)\to F(0)=A$. So, we get a contradiction.

\smallskip
{\bf Case 2.} Let $0$ be not a Picard exceptional value of $F$. Note that
\begin{align}\label{rj1.6} 
F^{\#}(\zeta)=\frac{\sqrt{\sum\limits_{j=1}^n|\partial_{\zeta_j}(F(\zeta))|^2}}{1+|F(\zeta)|^2},\;\;\zeta\in\mathbb{C}^n
\end{align}
and $F^{\#}(\zeta)\leq F^{\#}(0)=1$ for all $\zeta\in\mathbb{C}^n$. Therefore from (\ref{rj1.6}), it is easy to conclude that $\partial_{\zeta_k}(F(\zeta))\not\equiv 0$ for atleast one $k\in\{1,2,\ldots,n\}$.

Let $F(\zeta_0)=0$. Then applying Hurwitz's theorem (see \cite{PVD3}) to (\ref{rj1.1}), we get a sequence $\{\zeta_j\}$ in $\mathbb{C}^n$ such that $\zeta_j\rightarrow \zeta_0$ and 
\begin{align*}
F_j(\zeta)=\rho_n^{-\frac{1}{2}}\left\lbrace f_j(z_j+\rho_j \zeta)-a\right\rbrace=0
\end{align*}
for sufficiently large values of $j$. This implies that $f_j(z_j+\rho_j \zeta)=a$. Since $f_j(z)=a\Rightarrow \partial_{z_k}(f_j(z))=a$, we have $\partial_{z_k}(f_j(z_j+\rho_j \zeta_j))=a$. Then from (\ref{rj1.3}), we have
\begin{align*}
\partial_{z_k}(F(\zeta_0))=\lim\limits_{j\to\infty} \rho_j^{\frac{1}{2}} a=0.
\end{align*}

Hence $F=0\Rightarrow \partial_{z_k}(F)=0$. 
Now from (\ref{rj1.3}), we see that 
\begin{align}\label{rj1.7} 
\rho_{j}\left(\partial_{z_k}(f_j(z_{j}+\rho_{j}\zeta))-b\right)\rightarrow \partial_{z_k}(F(\zeta)).
\end{align}

Let $\partial_{z_k}(F(\xi_{0}))=0$. Then by (\ref{rj1.7}) and Hurwitz's theorem, there exists a sequence $\{\xi_{j}\}\subset\mathbb{C}^n$, $\xi_{j}\rightarrow \xi_{0}$ such that (for sufficiently large $j$) $\partial_{z_k}(f_j(z_{j}+\rho_{j}\xi_{j}))=b$. 
By the given condition, we have $f_j(z_{j}+\rho_{j}\xi_{j})=b$. Therefore from (\ref{rj1.1}), we have
\begin{align}\label{rj1.7}
F_j(\xi_{j})=\rho_j^{-\frac{1}{2}}\left\lbrace f_j(z_j+\rho_j \xi_j)-a\right\rbrace=(b-a)\rho_j^{-\frac{1}{2}}.
\end{align}

Since $a\neq b$ and $\rho_j\to 0$, from (\ref{rj1.7}), we get
\begin{align*}
F(\xi_0)=\lim\limits_{j\to \infty} F_j(\xi_{j})=\infty
\end{align*}
which contradicts the fact that $\partial_{z_k}(F(\xi_{0}))=0$. Hence $\partial_{z_k}(F)$ has no zeros. Since $F=0\Rightarrow \partial_{z_k}(F)=0$, we conclude that $0$ is a Picard exceptional value of $F$, which is impossible.

Hence $\mathcal{F}$ is normal at $z_0$. Consequently $\mathcal{F}$ is normal on $\Omega$. 
\end{proof}

For the concavity of logarithmic function, we have the following lemma.
\begin{lem}\cite[Lemma 1.32]{Hu-Li-Yang-2003} \label{ln2.7} Take $r>0$. Let $h$ be a non-negative function on $\mathbb{C}^n(r)$ such that $\log^+ h$ is integrable over $\mathbb{C}^n(r)$. Then 
\[\mathbb{C}^n[r;\log^+h]\leq \log^+(\mathbb{C}^n[r;h])+\log 2.\]
\end{lem}

In 1995, Ye \cite{27} obtained the following result.
\begin{lem}\cite[Lemma 4]{27} \label{ln2.8} Let $f$ be a non-constant meromorphic function in $\mathbb{C}^n$. Then for any $0<\alpha<\frac{1}{2}$, there is a constant $C>1$ such that for any $r_0<r<R$ and any $j\in\{1,2,\ldots,n\}$ we have 
\[\mathbb{C}^n\bigg\langle r;\bigg|\frac{\partial_{z_j}(f)}{f}\bigg|^{\alpha}\bigg\rangle\leq C\left\{\left(\frac{R}{r}\right)^{2n-1}\frac{T(R,f)}{R-r}\right\}^{\alpha}.\]
\end{lem}

\begin{lem}\label{ln2.9} Let $f$ be a non-constant meromorphic function on $\mathbb{C}^n$ such that $\rho(f)\leq 1$. Then any $j\in\{1,2,\ldots,n\}$, we have
\begin{align*}
m\left(r,\frac{\partial_{z_j}(f)}{f}\right)=o(\log r)\;\;\text{as}\;\;r\to \infty.
\end{align*}
\end{lem}
\begin{proof} By the definition of the proximity function, we have
\begin{align}\label{aa.1}
 m\left(r,\frac{\partial_{z_j}(f)}{f}\right)=\mathbb{C}^n\bigg\langle r;\log^+\bigg|\frac{\partial_{z_j}(f)}{f}\bigg|\bigg\rangle,
 \end{align}
where $j\in\{1,2,\ldots,n\}$. Now using Lemmas \ref{ln2.7} and \ref{ln2.8} to (\ref{aa.1}), we obtain 
\begin{align}\label{aa.2} 
m\left(r,\frac{\partial_{z_j}(f)}{f}\right)&\leq 
\frac{1}{\alpha}\log^+ \mathbb{C}^n\bigg\langle r;\log^+\bigg|\frac{\partial_{z_j}(f)}{f}\bigg|^{\alpha}\bigg\rangle+O(1)\nonumber\\
&\leq
 \log^+\left(\left(\frac{R}{r}\right)^{2n-1}\frac{T(R,f)}{R-r}\right)+O(1),
 \end{align}
where $R>r>r_0$ and $j\in\{1,2,\ldots,n\}$. Again by the given condition for every $\varepsilon>0$, there exists $r_1>0$ such that 
\begin{align}\label{aa.3} 
T(r,f)<r^{\rho+\varepsilon},\;\forall \;\; r>r_1.
\end{align}

Let us choose $R=2r$, where $r>\max\{r_0,r_1\}$. Then from (\ref{aa.2}) and (\ref{aa.3}), we get 
\begin{align}\label{aa.4} 
m\left(r,\frac{\partial_{z_j}(f)}{f}\right)\leq
\log^+(2^{2n-1}3^{\rho_1+\varepsilon}r^{\rho+\varepsilon-1})+O(1)\leq
 \log^+(r^{\rho+\varepsilon-1})+O(1),
 \end{align}
where $j\in\{1,2,\ldots,n\}$. 

If $\rho<1$, then choosing $\varepsilon>0$ in such a way that $\rho+\varepsilon-1<0$, we get from (\ref{aa.4}) that $m\left(r,\frac{\partial_{z_j}(f)}{f}\right)=O(1)$ and so
\begin{align*}
m\left(r,\frac{\partial_{z_j}(f)}{f}\right)=o(\log r)\;\;\text{as}\;\;r\to \infty,
\end{align*}
where $j\in\{1,2,\ldots,n\}$.
Next suppose that $\rho=1$. Then from (\ref{aa.4}), we get
\bea\label{aa.5} m\left(r,\frac{\partial_{z_j}(f)}{f}\right)&\leq & \log^+(2^{2n-1}3^{1+\varepsilon}r^{\varepsilon})+O(1)
 \leq \log^+(r^{\varepsilon})+O(1)=\varepsilon \log r+O(1).\eea
for sufficiently large $r$. Since $\varepsilon>0$ is arbitrary, we get from (\ref{aa.5}) that
\begin{align*}
 m\left(r,\frac{\partial_{z_j}(f)}{f}\right)=o(\log r)\;\;\text{as}\;\;r\to\infty
 \end{align*}
 for $j\in\{1,2,\ldots,n\}$.
\end{proof}

We now recall the following lemma.

\begin{lem}\cite[Lemma 2.2.9]{NW}\label{ln2.5a} Let the notation be as above.
\begin{enumerate}
\item[(i)] If $l=\dim A < n$, then Lebesgue measure of $A$ in $\mathbb{C}^n$ is zero.
\item [(ii)] $\int_{K\cap R(A)} \alpha^l < \infty$.
\end{enumerate}
\end{lem}

Let $A \subset  \mathbb{C}^n$ be an analytic subset of pure dimension $l<n$. Consider the collection $\mathscr{K}$ of all compact sets in $\mathbb{C}^n$. Then
\begin{align*}
\displaystyle m^*_n(A)=\inf\left\lbrace \sideset{}{_{j=1}^{\infty}}{\sum} \int_{K_j\cap A} \alpha ^l: K_1,K_2,\ldots\in \mathscr{K}\;\;\text{so that} A\subset \sideset{}{_{j=1}^{\infty}}{\bigcup} K_j \right\rbrace,
\end{align*}
which is called the Lebesgue outer measure of $A$ in $\mathbb{C}^n$. Clearly by Lemma \ref{ln2.5a}, we have $m^*_n(A)=0$. Also by the definition of $m^*_n$, for a given $\delta>0$ there exists a countable collection $\{K_j\}$ of compact sets such that $A\subset \sideset{}{_{j=1}^{\infty}}{\bigcup} K_j$ satisfying $\sideset{}{_{j=1}^{\infty}}{\sum} \int_{K_j\cap A} \alpha ^l<\delta$.
Since $A$ is a closed set, it follows that $K_j\cap A$ is also a compact set for $j=1,2,\ldots$ and so the set $K_j\cap A$ is also measurable for $j=1,2,\ldots$. Again since the integral $\int_{K_j\cap A} \alpha ^l$ is considered as a measure of $K\cap A$, we have $\sideset{}{_{j=1}^{\infty}}{\sum} m(K_j\cap A)<\delta$.

\medskip
Let us take $K_j=\ol B(a_j,r_j)$, where $a_j\in\mathbb{C}^n\backslash \{0\}$ and $r_j>0$ for $j=1,2,\ldots$. Choose $a_j$ and $r_j$ in such a way that $r_j<||a_j||\delta^j$ for $j=1,2,\ldots$ and $\lim\limits_{j\to\infty}||a_j||=\infty$, where $0<\delta<1$.
Let
\begin{align*}
 \displaystyle E=\sideset{}{_{j=1}^{\infty}}{\bigcup} \ol B(a_j,r_j).
 \end{align*}

Clearly $\lim\limits_{j\to\infty}||a_j||=\infty$ and $\sideset{}{_{j=1}^{\infty}}{\sum}\frac{r_j}{||a_j||}<\infty$. Here $E$ is called an $\varepsilon$-set in $\mathbb{C}^n$. Obviously $A\subset E$ (see \cite{Majumder-Sarkar-2026}).

Let $S=\{\|z\|=r: z\in E,\quad \|z\|>1\}$, where $E=\sideset{}{_{j=1}^{\infty}}{\bigcup}B_j=\sideset{}{_{j=1}^{\infty}}{\bigcup} \ol B(a_j,r_j)$ is an $\varepsilon$-set in $\mathbb{C}^n$. Then the logarithmic measure $l_j$ of the set of $r$ corresponding to $\|z\|=r$  which $B_j$ meets is given by
\begin{align*} l_j=\int\limits_{\|a\|-r_j}^{\|a\|+r_j}\frac{d r}{r}=\log \frac{\|a\|+r_j}{\|a\|-r_j}<3\frac{r_j}{\|a_j\|}, \quad \text{if}\;r_j<\frac{1}{2}\|a_j\|.
\end{align*}

Obviously $\sideset{}{_{j=1}^{\infty}}{\sum} l_j<+\infty$.

\medskip
Let $f$ be a holomorphic function in $\|z\|\leq R$, where $R>0$. For $0\leq r\leq R$, we define 
\begin{align*}
M(r,f)=\max\limits_{||z||=r}|f(z)|.
\end{align*}

The function $M(r,f)$ is called the growth function of $f$. Obviously $M(r,f)$ is steadily increasing function of $r$ and for a non-constant holomorphic function $f$ in $\mathbb{C}^n$, we have $M(r,f)\to \infty$ as $r\to \infty$. Also by the maximum modulus principle, any holomorphic function $f(z)$ in $\|z\|\leq r$ attains its maximum modulus at a point of $\|z\|=r$.

\begin{lem}\label{ln2.10}\cite{Majumder-Sarkar-2026} Let $f$ be a transcendental meromorphic function in $\mathbb{C}^n$ of finite order $\rho$. 
Then there exists an $\varepsilon$-set $E$ in $\mathbb{C}^n$ such that
\beas \displaystyle \left|\frac{\partial_{z_i}(f(z))}{f(z)}\right|\leq ||z||^{\rho-1+\delta},\eeas
holds for all large values of $||z||$, where $z\in \mathbb{C}^n\backslash E$ and $\delta>0$ is a given constant. 
\end{lem}

\begin{lem}\label{ln2.11} Let $f$ be a non-constant entire function in $\mathbb{C}^n$ of finite order. For $a\in\mathbb{C}^n$, if $f=a\rightleftharpoons \partial_{z_i}(f)=a$, where $i\in\{1,2,\ldots,n\}$, then $\partial_{z_i}(f)-a=A_i(f-a)$ for some non-zero constant $A_i$ in $\mathbb{C}$.
\end{lem}
\begin{proof} Note that $f=a\rightleftharpoons \partial_{z_i}(f)=a$, where $f$ is a non-constant entire function in $\mathbb{C}^n$ of finite order $\rho$. Using Lemma \ref{ln2.6}, we deduce that
\begin{align}\label{ln11.0}
\frac{\partial_{z_i}(f)-a}{f-a}=e^{P_i},
\end{align}
where $P_i$ is a polynomial in $\mathbb{C}^n$. We want to prove that $P_i$ is a constant. Suppose on contrary that $P_i$ is non-constant and let $\deg(P_i)=m_i$. Set
\begin{align*}
P_i(z)=\sum\limits_{\alpha_1,\ldots,\alpha_n=0}^{m_i}a_{\alpha_1\ldots\alpha_n}z_1^{\alpha_1}\ldots z_n^{\alpha_n}.
\end{align*}

Define
\begin{align*}
||a_{|\alpha|}||_{1}=\sum\limits_{\alpha_1+\ldots+\alpha_n=|\alpha|} |a_{\alpha_1\ldots\alpha_n}|.
\end{align*}

Now one can easily prove that (see the proof of Satz 1.2 \cite{Jank-Volkmann-1985})
\begin{align}\label{ln11.0a}
(1-o(1))||a_{|\alpha|}||_1r^{m_i}\leq |P_i(z)|\leq (1+o(1))||a_{|\alpha|}||_1r^{m_i}
\end{align}
for large values of $r=\|z\|$, where $|\alpha|=m_i$.

If we take $g=f-a$, then from (\ref{ln11.0}), we obtain
\begin{align}\label{ln11.3}
\frac{\partial_{z_i}(g)}{g}-\frac{a}{g}=e^P.
\end{align}

Let $A=\operatorname{supp}\mu^0_g$. Clearly $A$ is either empty or an analytic set of pure dimension $n-1$. Then the Lebesgue measure of $A$ in $\mathbb{C}^n$ is zero. Now by Lemma \ref{ln2.10}, there exists an $\varepsilon$-set $E$ in $\mathbb{C}^n$ such that $A\subset E$ and
\begin{align}\label{ln11.1} 
\left|\frac{\partial_{z_i}(g(z))}{g(z)}\right|\leq ||z||^{\rho-1+\delta}
\end{align}
holds for all large values of $||z||$, where $z\in \mathbb{C}^n\backslash E$ and $\delta>0$ is a given constant.
Let $S=\{\|z\|=r: z\in E,\quad \|z\|>1\}$. Then $S$ is of finite logarithmic measure. Choose a sequence $\{z_k\}$ such that $\|z_k\|=r_k$, $r_k\not\in [0,1]\cup S$ and $|g(z_k)|=M(r_k,g)$, $r_k\to +\infty$ as $k\to +\infty$. Clearly
\begin{align}\label{ln11.2} 
\lim\limits_{k\rightarrow +\infty}\frac{|a|}{|g(z_k)|}=\lim\limits_{k\rightarrow +\infty}\frac{|a|}{M(r_k,g)}=0,\;\;\text{i.e.,}\;\;\frac{|a|}{|g(z_k)|}=o(1)
\end{align}
for sufficiently large $\|z_k\|=r_k\not\in [0,1]\cup S$. 

Taking the principal branch of the logarithm, we deduce from (\ref{ln11.3}) that
\begin{align}\label{ln11.4}
P(z)=\log \left( \frac{\partial_{z_i}(g(z))}{g(z)} - \frac{a}{g(z)} \right)
=\log \left| \frac{\partial_{z_i}(g(z))}{g(z)}-\frac{a}{g(z)}\right|+ i\arg \left(\frac{\partial_{z_i}(g(z))}{g(z)}-\frac{a}{g(z)}\right).
\end{align}

Substituting (\ref{ln11.0a}), (\ref{ln11.1}) and (\ref{ln11.2}) into (\ref{ln11.4}), we get
\begin{align*} ||a_n||_1 r_k^{m_i}(1 - o(1))\leq |P(z_k)|&\leq \log \left|\frac{\partial_{z_i}(g(z_k))}{g(z_k)}\right|+\log \left(\frac{|a|}{|g(z_k)|}+e\right)+O(1)\\&\leq (\rho-1+\delta)\log r_k+O(1),
\end{align*}
which is impossible. Hence $P_i$ is a constant. If we take $A_i=e^{P_i}$, then from (\ref{ln11.0}), we have $\partial_{z_i}(f)-a=A_i(f-a)$ for some non-zero constant $A_i$ in $\mathbb{C}$.
\end{proof}

\begin{lem}\label{ln2.5}\cite[Lemma 1.2]{Hu-Yang-1996} Let $f$ be a non-constant meromorphic function in $\mathbb{C}^n$ and let $a_1,a_2,\ldots,a_q$ be different points in $\mathbb{C}\cup\{\infty\}$. Then
\begin{align*}
(q-2)T(r,f)\leq \sideset{}{_{j=1}^{q}}{\sum} \ol N(r,a_j;f)+O(\log (rT(r,f)))
\end{align*}
holds only outside a set of finite measure on $\mathbb{R}^+$.
\end{lem}

\begin{lem}\label{ln2.12}\cite{Majumder-Sarkar-2027} Let $g(z)$ be a meromorphic function in $\mathbb{C}^n$. If $\partial^2_{z_i}(g(z))\equiv 0$ for all $i=1,2,\ldots,n$, then $g(z)$ must be a polynomial in $\mathbb{C}^n$.
\end{lem}

\section{\bf{Proofs of the main theorems}}

\begin{proof}[{\bf Proof of Theorem \ref{t1.1}}] 
By the given conditions, we have $f=a\Rightarrow \partial_{z_i}(f)=a$ and $f=b\Leftrightarrow \partial_{z_i}(f)=b$ for all $i=1,2,\ldots,n$.  Set $\mathcal{F}=\{f_{\omega}\}$, where $f_{\omega}(z)=f(\omega+z)$, $\omega\in\mathbb{C}^n$. Clearly $\mathcal{F}$ is a family of holomorphic functions defined on $\mathbb{C}^n$. By assumption for any function $g(z)=f(\omega+z)$, we have $g=a\Rightarrow \partial_{z_i}(g)=a$ and $g=b\Leftrightarrow  \partial_{z_i}(g)=b$ for all $i=1,2,\ldots,n$. Since $\mathcal{F}$ is a family of holomorphic functions in $\Omega\subset\mathbb{C}^n$, then by Lemma \ref{ln2.3}, we conclude that $\mathcal{F}$ is normal on $\Omega$. 
Now by Lemma \ref{ln2.1}, we have $f^{\#}(\omega)=f^{\#}_{\omega}(0)\leq M$ for some $M>0$ and for all $\omega\in\mathbb{C}^n$. Therefore Lemma \ref{ln2.3} gives us $\rho(f)\leq 1$. 
We now consider the following three cases.

\smallskip
{\bf Case 1.} Let $a$ be a Picard exceptional value of $f$. Then by Lemma \ref{ln2.6}, we assume that
\begin{align*}
f(z)-a=e^{P(z)},
\end{align*}
where $P(z)$ is a polynomial in $\mathbb{C}^n$ such that $\deg(P)=1$. Clearly $\partial_{z_i}(f(z))=\partial_{z_i}(P(z))e^{P(\zeta)}$ for $i=1,2,\ldots,n$. In this case, using Lemma \ref{ln2.5}, we deduce that $b$ is not a Picard exceptional values of $f$. Let $z_0\in\mathbb{C}^n$ be a zero of $f-b$. Since $f=b\Leftrightarrow \partial_{z_i}(f)=b$, we have $f(z_0)=a+e^{P(z_0)}=b$ and $\partial_{z_i}(f(z_0))=\partial_{z_i}(P(z_0))e^{P(z_0)}=b$ for $i=1,2,\ldots,n$. Now eliminating $e^{P(z_0)}$, we get $\partial_{z_i}(P)=\frac{b}{b-a}$ for $i=1,2,\ldots,n$. Therefore, we have 
\begin{align*}
f(z)=a+ce^{\frac{b}{b-a}(z_1+z_2+\ldots+z_n)},\;\;c\in\mathbb{C}\backslash \{0\}.
\end{align*}

\smallskip
{\bf Case 2.} Let $b$ be a Picard exceptional value of $f$. Clearly $a$ is not a Picard exceptional values of $f$. Now by Lemma \ref{ln2.6}, we may assume that
\begin{align*}
f(z)-b=e^{P(z)},
\end{align*}
where $P(z)$ is a polynomial in $\mathbb{C}^n$ such that $\deg(P)=1$. In this case also, one can easily obtain $\partial_{z_i}(P)=\frac{a}{a-b}$ for $i=1,2,\ldots,n$. Obviously $a\neq 0$. Therefore, we have 
\begin{align*}
f(z)=b+ce^{\frac{a}{b-a}(z_1+z_2+\ldots+z_n)},\;\;c\in\mathbb{C}\backslash \{0\}.
\end{align*}

\smallskip
{\bf Case 3.} Let both $a$ and $b$ be not Picard exceptional values of $f$. 
Set
\bea\label{e3.1} \Phi_i(z)=\frac{(f(z)-\partial_{z_i}(f(z)))\partial_{z_i}(f(z))}{(f(z)-a)(f(z)-b)},\eea
where $i=1,2,\ldots,n$ and $z\in\mathbb{C}^n$.
We now consider the following two sub-cases.\par

\smallskip
{\bf Sub-case 3.1.} Suppose $\Phi_i\not\equiv 0$ for atleast one $i\in\{1,2,\ldots,n\}$. For the sake of simplicity we may assume that $\Phi_k\not\equiv 0$. Then $f\not\equiv \partial_{z_k}(f)$. Since $f=a\Rightarrow \partial_{z_k}(f)=a$ and $f=b\Leftrightarrow \partial_{z_k}(f)=b$, one can easily prove that $\Phi_k$ is a holomorphic function in $\mathbb{C}^n$. 
Now from (\ref{e3.1}), we have 
\begin{align}\label{e3.2} 
\Phi_k(z)=\frac{1}{a-b}\left(a\frac{\partial_{z_k}(f(z))}{f(z)-a}-b\frac{\partial_{z_k}(f(z))}{f(z)-b}\right)\left(1-\frac{\partial_{z_k}(f(z))}{f(z)}\right).
\end{align}

Therefore using Lemma \ref{ln2.9} to (\ref{e3.2}), we get $m(r,\Phi_k)=o(\log r)$ as $r\to\infty$. Since $\Phi_k(z)$ is a holomorphic function in $\mathbb{C}^n$, we get $T(r,\Phi_k)=o(\log r)$ as $r\to\infty$, which shows that $\Phi_k(z)$ is a non-zero constant, say $c_k\in\mathbb{C}\backslash \{0\}$. Therefore from (\ref{e3.1}), we get
\begin{align}\label{e3.2a} 
(f(z)-\partial_{z_k}(f(z)))\partial_{z_k}(f(z))\equiv c_k(f(z)-a)(f(z)-b).
\end{align}

Following two sub-cases are immediately.\par

\smallskip
{\bf Sub-case 3.1.1.} Let $a\neq 0$. Since $f=a\Rightarrow \partial_{z_k}(f)=a$ and $f=b\Leftrightarrow \partial_{z_k}(f)=b$, it follows that both $f-a$ and $f-b$ have only simple zeros. Consequently from (\ref{e3.1}), we deduce that $\partial_{z_k}(f)$ has no zeros. Note that $\rho(\partial_{z_k}(f))\leq 1$ and so we may take 
\begin{align}\label{e3.3a}
\partial_{z_k}(f(z))=e^{Q(z)},
\end{align} where $Q(z)$ is a polynomial in $\mathbb{C}^n$ such that $\deg(Q)=1$. Therefore it is easy to prove that $\partial_{z_k}(f)-b$ has only simple zeros and so $f=b\rightleftharpoons \partial_{z_k}(f)=b$. Now by Lemma \ref{ln2.11}, we have
\begin{align}\label{e3.3b}
\partial_{z_k}(f)-b=A_k(f-b)
\end{align}
for some non-zero constant $A_k$ in $\mathbb{C}$. Note that $a$ is not a Picard exceptional value of $f$. Let $z_0$ be a zero of $f-a$. Since $f=a\Rightarrow \partial_{z_k}(f)=a$, we have $\partial_{z_k}(f(z_0))=a$. Then from (\ref{e3.3b}), we get $a-b=A_k(a-b)$. Since $a\neq b$, we have $A_k=1$ and so (\ref{e3.3b}) yields $f\equiv \partial_{z_k}(f)$, which is impossible. 

\smallskip
{\bf Sub-case 3.1.2.} Let $a=0$. Now from (\ref{e3.2a}), we have
\begin{align}\label{e3.4a} 
(f(z)-\partial_{z_k}(f(z)))\partial_{z_k}(f(z))\equiv c_k f(z)(f(z)-b).
\end{align}

Let $z_0\in\mathbb{C}^n$ be a simple zero of $f$. Since $f=0\Rightarrow \partial_{z_k}(f)=0$, it follows that $z_0$ is a zero of $\partial_{z_k}(f)$. Then from (\ref{e3.4a}), we immediately get a contradiction. Hence $f$ has no simple zero.
Again by a routine calculation, one can easily conclude from (\ref{e3.4a}) that all the zeros of $f(z)$ have multiplicity exactly $2$.  Consequently, we may take $f=h^2$, where $h$ is a holomorphic function in $\mathbb{C}^n$ such that $h$ has only simple zeros. Note that $\partial_{z_k}(f)=2h\partial_{z_k}(h)$. In this case, we have
$h^2=0\Rightarrow 2h\partial_{z_k}(h)=0$ and $h^2=b\Leftrightarrow 2h\partial_{z_k}(h)=b$. Also from (\ref{e3.4a}), we get
\begin{align}\label{e3.5} 
2(h(z)-2\partial_{z_k}(h(z)))\partial_{z_k}(h(z))\equiv c_k(h^2(z)-b)
\end{align}
for all $z\in\mathbb{C}^n$. Now from (\ref{e3.5}), we deduce that $\partial_{z_k}(h)$ has no zeros. Note that $\rho(\partial_u(h))\leq 1$ and so we may take $\partial_{z_k}(h(z))=e^{R(z)}$, where $R(z)$ is a polynomial in $\mathbb{C}^n$ such that $\deg(R)=1$. Again from (\ref{e3.5}), we get
\begin{align*}
 2h(z)e^{R(z)}-4e^{2R(z)}\equiv c_k(h^2(z)-b)
 \end{align*}
and so
\begin{align}\label{e3.6} 
(\partial_{z_k}(R(z))-c_k)h(z)\equiv (4\partial_{z_k}(R(z))-1)e^{R(z)}.
\end{align}

Since $0$ is not a Picard exception value of $h$, one can easily conclude from (\ref{e3.6}) that $\partial_{z_k}(R(z))-c_k=0$ and $4\partial_{z_k}(R(z))-1=0$. Consequently, we have 
\begin{align*}
\partial_{z_k}(R(z))=c_k=\frac{1}{4}
\end{align*}
and so from (\ref{e3.5}), we have 
\begin{align*}
(h(z)-4\partial_{z_k}(h(z)))^2\equiv b.
\end{align*}

Set $h(z)-4\partial_{z_k}(h(z))=d$, where $d$ is a root of the equation $x^2=b$. Since $f=h^2$ and $\partial_{z_k}(h(z))=e^{R(z)}$, we deduce that 
\begin{align*}
f(z)=16e^{2R(z)}+8de^{R(z)}+b.
\end{align*}

Clearly
\begin{align}\label{e3.7}
f(z)-b=16e^{2R(z)}+8de^{R(z)}=8(2e^{R(z)}+d)e^{R(z)}
\end{align}
and 
\begin{align}\label{e3.8}
\partial_{z_k}(f(z))-b=8e^{2R(z)}+2de^{R(z)}-d^2.
\end{align}

Since $f=b\Leftrightarrow \partial_{z_k}(f)=b$, from (\ref{e3.7}) and (\ref{e3.8}), we immediately get a contradiction.\par

\smallskip
{\bf Sub-case 3.2.} Let $\Phi_i\equiv 0$ for all $i=1,2,\ldots,n$. Since $\partial_{z_i}(f)\not\equiv 0$, it follows from (\ref{e3.1}) that 
\begin{align}\label{e3.9}
f(z)\equiv \partial_{z_i}(f(z))
\end{align} 
for all $i=1,2,\ldots,n$. Then from (\ref{e3.9}), we conclude that both $f$ and $\partial_{z_i}(f)$ have no zeros for all $i=1,2,\ldots,n$. If we take $f=e^{\zeta},$ where $\zeta$ is a non-constant entire function in $\mathbb{C}^n$, then from (\ref{e3.9}), we get $\partial_{z_i}(\zeta)=1$ for all $i=1,2,\ldots,n$.
By Taylor expansion, we have
\begin{align}\label{e3.10} 
\zeta(z)=\sum\limits_{\mu_1,\ldots,\mu_n=0}^{\infty} b_{\mu_1\ldots\mu_n} z_1^{\mu_1}\ldots z_n^{\mu_n},
\end{align}
where the coefficient $b_{\mu_1\ldots\mu_n}$ is given by
\begin{align}\label{e3.11} 
b_{\mu_1\ldots\mu_n}=\frac{1}{\mu_1!\ldots \mu_n!}\frac{\partial^{|I|}\zeta(0,0,\ldots,0)}{\partial z_1^{\mu_1}\cdots \partial z_n^{\mu_n}}.
\end{align}

Now from (\ref{e3.10}) and (\ref{e3.11}), we have $\zeta(z)=z_1+z_2+\ldots+z_n+A_0$, where $A_0=\zeta(0,0,\ldots,0)$. Therefore
\begin{align*}
f(z_1,\ldots,z_n)=ce^{z_1+z_2+\ldots+z_n},\;\;c=\exp(A_0).
\end{align*}
\end{proof}

\begin{proof}[{\bf Proof of Theorem \ref{t1.2}}] 
By the given conditions, we have $\partial_{z_i}(f)=a\Rightarrow f=a$ and $f=b\Leftrightarrow \partial_{z_i}(f)=b$ for all $i=1,2,\ldots,n$.  Set $\mathcal{F}=\{f_{\omega}\}$, where $f_{\omega}(z)=f(\omega+z)$, $\omega\in\mathbb{C}^n$. Clearly $\mathcal{F}$ is a family of holomorphic functions defined on $\mathbb{C}^n$. By assumption for any function $g(z)=f(\omega+z)$, we have $\partial_{z_i}(g)=a\Rightarrow g=a$ and $g=b\Leftrightarrow  \partial_{z_i}(g)=b$ for all $i=1,2,\ldots,n$. Since $\mathcal{F}$ is a family of holomorphic functions in $\Omega\subset\mathbb{C}^n$, then by Lemma \ref{ln2.3a}, we conclude that $\mathcal{F}$ is normal on $\Omega$. Now by Lemma \ref{ln2.1}, we have $f^{\#}(\omega)=f^{\#}_{\omega}(0)\leq M$ for some $M>0$ and for all $\omega\in\mathbb{C}^n$. Therefore Lemma \ref{ln2.3} gives us $\rho(f)\leq 1$. We put 
\begin{align}
\label{e4.1} \varphi_{ji}= \frac{\partial_{z_jz_i}^2(f)(\partial_{z_i}(f)-f)}{(f-b)(\partial_{z_i}(f)-a)}.
\end{align}

Now we consider the following two cases.

\smallskip
{\bf Case 1.} Suppose $\varphi_{st}\not\equiv 0$, where $s,t\in\{1,2,\ldots,n\}$. Then $f\not\equiv \partial_{z_t}(f)$ and $\partial_{z_sz_t}^2(f)\not\equiv 0$. Since $\partial_{z_t}(f)=a\Rightarrow f=a$ and $f=b\Leftrightarrow \partial_{z_t}(f)=b$, one can easily prove that $\varphi_{st}$ is a holomorphic function in $\mathbb{C}^n$. 
Now from (\ref{e4.1}), we have
\begin{align}\label{e4.2} 
\varphi_{st}=\frac{\partial_{z_t}(f)}{f-b}\frac{\partial_{z_sz_t}^2(f)}{\partial_{z_t}(f)-a}-\frac{\partial_{z_sz_t}^2(f)}{\partial_{z_t}(f)-a}-b\frac{\partial_{z_t}(f)}{f-b}\frac{\partial_{z_sz_t}^2(f)}{(\partial_{z_t}(f)-a)\partial_{z_t}(f)}.
\end{align}

Therefore using Lemma \ref{ln2.9} to (\ref{e4.2}), we get
\begin{align*}
 T(r,\varphi_{st})=m(r,\varphi_{st})&\leq m\left(r,\frac{\partial_{z_t}(f)}{f-b}\right)+m\left(r,\frac{\partial_{z_sz_t}^2(f)}{\partial_{z_t}(f)-a}\right)+m\left(r,\frac{\partial_{z_sz_t}^2(f)}{\partial_{z_t}(f)-a}\right)\\&+m\left(r,\frac{\partial_{z_t}(f)}{f-b}\right)+m\left(r,\frac{\partial_{z_sz_t}^2(f)}{(\partial_{z_t}(f)-a)\partial_{z_t}(f)}\right)+O(1)=o(\log r)
\end{align*}
as $r\to\infty$ and so $\varphi_{st}$ is a constant. Suppose $\varphi_{st}=c_{st}\in\mathbb{C}\backslash \{0\}$. Now from (\ref{e4.1}), we have
\begin{align}
\label{e4.3} c_{st}=\frac{\partial_{z_sz_t}^2(f)(\partial_{z_t}(f)-f)}{(f-b)(\partial_{z_t}(f)-a)}.
\end{align}

Since $\partial_{z_t}(f)=a\Rightarrow f=a$ and $f=b\Leftrightarrow \partial_{z_t}(f)=b$, it follows from (\ref{e4.3}) that $\partial_{z_sz_t}^2(f)$ has no zeros. Consequently both $f-b$ and $\partial_{z_t}(f)-b$ have only simple zeros. Hence $f=b\rightleftharpoons \partial_{z_k}(f)=b$. Now by Lemma \ref{ln2.11}, we have
\begin{align}\label{e4.4}
\partial_{z_t}(f)-b=A_t(f-b)
\end{align}
for some non-zero constant $A_t$.

\smallskip
First we suppose that $a$ is not a Picard exceptional of $\partial_{z_t}(f)-a$. Let $z_0$ be a zero of $\partial_{z_t}(f)-a$. Since $\partial_{z_t}(f)=a\Rightarrow f=a$, we have $f(z_0)=a$ and so from (\ref{e4.4}), we obtain $a-b=A_t(a-b)$. Since $a\neq b$, we get $A_t=1$ and so (\ref{e4.4}) yields $f\equiv \partial_{z_t}(f)$, which is impossible.
 
\smallskip
Next we suppose that $a$ is a Picard exceptional of $\partial_{z_t}(f)-a$. So we may take 
\begin{align}\label{e4.5}
\partial_{z_t}(f)=e^{Q_t}+a,
\end{align} 
where $Q_t(z)$ is a polynomial in $\mathbb{C}^n$ such that $\deg(Q_t)=1$. Now substituting (\ref{e4.5}) into (\ref{e4.4}), we deduce that
\begin{align}\label{e4.6}
f=\frac{1}{A_t}e^{Q_t}+b+\frac{a-b}{A_t}.
\end{align}

It is clear from (\ref{e4.6}) that 
\begin{align}\label{e4.7}
\partial_{z_t}(f)=\frac{\partial_{z_t}(Q_t)}{A_t}e^{Q_t}.
\end{align}

Now comparing (\ref{e4.5}) and (\ref{e4.6}), we get
\begin{align*}
\left(1-\frac{\partial_{z_t}(Q_t)}{A_t}\right)e^{Q_t}=-a.
\end{align*}

Since $a\neq 0$, we get a contradiction from above.

\smallskip
{\bf Case 2.} Suppose $\varphi_{ji}\equiv 0$ for all $i,j\in\{1,2,\ldots,n\}$. Then from (\ref{e4.1}), we have either $f=\partial_{z_i}(f)$ for all $i=1,2,\ldots,n$ or $\partial^2_{z_jz_i}(f)\equiv 0$ for all $i,j\in\{1,2,\ldots,n\}$. If $\partial^2_{z_jz_i}(f)\equiv 0$ for all $i,j\in\{1,2,\ldots,n\}$, then by Lemma \ref{ln2.12}, we get a contradiction. Hence $f=\partial_{z_i}(f)$ for all $i=1,2,\ldots,n$. In this case also, we have
\begin{align*}
f(z_1,\ldots,z_n)=ce^{z_1+z_2+\ldots+z_n}, \;\;c\in\mathbb{C}\backslash \{0\}.
\end{align*}

\end{proof}

\begin{proof}[{\bf Proof of Theorem \ref{t1.3}}] 
By the given conditions, we have $\partial_{z_i}(f)=0\Rightarrow f=0$ and $f=b\rightleftharpoons \partial_{z_i}(f)=b$ for all $i=1,2,\ldots,n$. Now proceeding in the same way as done in the proof of Theorem \ref{t1.2}, we can easily prove that $\rho(f)\leq 1$.
Then by Lemma \ref{ln2.11}, we have
\begin{align}\label{e5.1}
\partial_{z_i}(f)-b=A_i(f-b),
\end{align}
where $A_i$ is a constant in $\mathbb{C}\backslash \{0\}$ for $i=1,2,\ldots,n$.
We now consider the following two cases.

\smallskip
{\bf Case 1.} Suppose $A_i=1$ for all $i=1,2,\ldots,n$. Then from (\ref{e5.1}), we have $f\equiv \partial_{z_i}(f)$ for $i=1,2,\ldots,n$.
In this case also, we have
\begin{align*}
f(z_1,\ldots,z_n)=ce^{z_1+z_2+\ldots+z_n},\;\;c\in\mathbb{C}\backslash \{0\}.
\end{align*}

\smallskip
{\bf Case 2.} Suppose $A_i\neq 1$ for atleast one $i\in\{1,2,\ldots,n\}$. For the sake of simplicity, we may assume that $A_t\neq 1$, where $t\in\{1,2,\ldots,n\}$.

First we suppose that $0$ is not a Picard exceptional of $\partial_{z_t}(f)$. Let $z_0$ be a zero of $\partial_{z_t}(f)$. Since $\partial_{z_t}(f)=a\Rightarrow f$, we have $f(z_0)=0$ and so from (\ref{e5.1}), we obtain $0-b=A_t(0-b)$. Since $0\neq b$, we get $A_t=1$, which is impossible.
 
\smallskip
Next we suppose that $0$ is a Picard exceptional of $\partial_{z_t}(f)$.
So we may take 
\begin{align}\label{e5.2}
\partial_{z_t}(f(z))=c_te^{a_{t1}z_1+a_{t2}z_2+\ldots+a_{tn}z_n},
\end{align} 
where $a_{t1},a_{t2},\ldots,a_{tn}$ and $c_t\neq 0$ are constants such that $(a_{t1},a_{2t},\ldots,a_{tn})\neq (0,0,\ldots,0)$. Now substituting (\ref{e5.2}) into (\ref{e5.1}), we deduce that
\begin{align}\label{e5.3}
f(z)=\frac{c_t}{A_t}e^{a_{t1}z_1+a_{t2}z_2+\ldots+a_{tn}z_n}+b-\frac{b}{A_t}.
\end{align}

If possible, suppose $0$ is not a Picard exceptional of $\partial_{z_i}(f)$, where $i=1,2,\ldots,n$ such that $i\neq t$. Then one can easily prove that $A_i=1$, where $i=1,2,\ldots,n$ such that $i\neq t$ and so $f\equiv \partial_{z_i}(f)$, where $i=1,2,\ldots,n$ such that $i\neq t$. In this case from (\ref{e5.3}), we get a contradiction. Hence $0$ is not a Picard exceptional of $\partial_{z_i}(f)$ for all $i=1,2,\ldots,n$. Consequently (\ref{e5.3}) holds for $t=1,2,\ldots,n$. For the uniqueness of $f(z)$, we may assume that
\begin{align}\label{e5.4}
f(z)=\frac{c}{A}e^{a_{1}z_1+a_{2}z_2+\ldots+a_{n}z_n}+b-\frac{b}{A},
\end{align}
where $a_{1},a_{2},\ldots,a_{n}$ and $c$ are non-zero constants. Also from (\ref{e5.2}), we have
\begin{align}\label{e5.5}
\partial_{z_i}(f(z))=ce^{a_{1}z_1+a_{2}z_2+\ldots+a_{n}z_n},
\end{align} 
where $i=1,2,\ldots,n$.
Again from (\ref{e5.4}), we get 
\begin{align}\label{e5.6}
\partial_{z_i}(f(z))=\frac{a_ic}{A}e^{a_1z_1+a_2z_2+\ldots+a_nz_n},
\end{align}
where $i=1,2,\ldots,n$.
Now comparing (\ref{e5.5}) and (\ref{e5.6}), we get
\begin{align*}
c\left(1-\frac{a_i}{A}\right)e^{a_1z_1+a_2z_2+\ldots+a_nz_n}=0,
\end{align*}
which shows that $a_i=A$ for $i=1,2,\ldots,n$. Finally from (\ref{e5.4}), we have 
\begin{align}\label{e5.4}
f(z)=\frac{c}{A}e^{A(z_1+z_2+\ldots+z_n)}+b-\frac{b}{A},
\end{align}
where $A$ and $c$ are non-zero constants.
\end{proof}

\medskip
{\bf Statements and declarations:}

\smallskip
\noindent \textbf {Conflict of interest:} The authors declare that there are no conflicts of interest regarding the publication of this paper.

\smallskip
\noindent{\bf Funding:} There is no funding received from any organizations for this research work.

\smallskip
\noindent \textbf {Data availability statement:}  Data sharing is not applicable to this article as no database were generated or analyzed during the current study.

\end{document}